\documentclass[11pt]{amsart}

\usepackage[a4paper,hmargin=3.5cm,vmargin=3.7cm]{geometry}
\usepackage{amssymb}
\usepackage{amsmath}
\usepackage{enumerate}
\usepackage[pdftex]{graphicx}
\usepackage{graphicx, color}
\usepackage[all]{xy}

\def\s{\mathbb{S}}

\def\R{\mathbb{R}}
\def\Z{\mathbb{Z}}
\def\Q{\mathbb{Q}}
\def\C{\mathbb{C}}
\def\H{\mathbb{H}}

\definecolor{alertmanzano}{rgb}{0.8,0,0.3}

\newcommand{\alertm}[1]{%
	\marginpar{%
		\ifodd\value{page} \raggedright \else \raggedleft \fi
		\footnotesize{\textcolor{alertmanzano}{#1}}
	}
}

\DeclareMathOperator{\sech}{sech}

\newtheorem{theorem}{Theorem}[section]

\newtheorem{corollary}[theorem]{Corollary}

\theoremstyle{definition}
\newtheorem{definition}[theorem]{Definition}
\newtheorem{remark}[theorem]{Remark}
\newtheorem{example}[theorem]{Example}

\numberwithin{equation}{section}

\begin{document}

\title[Helicoidal minimal surfaces in $\s^3$: An approach via spherical curves]{Helicoidal minimal surfaces in the 3-sphere: 
	\\ An approach via spherical curves}

\author[I. Castro]{Ildefonso Castro}
\address{Departamento de Matem\'{a}ticas \\
Universidad de Ja\'{e}n \\
23071 Ja\'{e}n, Spain and Instituto de Matem\'{a}ticas, Universidad de Granada (IMAG) \\
Orcid 0000-0003-3853-4967.
} 
\email{icastro@ujaen.es}

\author[I. Castro-Infantes]{Ildefonso Castro-Infantes}
\address{Departamento de Matem\'{a}ticas  \\
Universidad de Murcia \\
30100 Murcia, Spain \\
Orcid 0000-0002-5379-0566.} 
\email{ildefonso.castro@um.es}

\author[J. Castro-Infantes]{Jes\'{u}s Castro-Infantes}
\address{Departamento de Matem\'atica Aplicada a las Tecnolog\'{\i}as de la Informaci\'on y las Comunicaciones \\
Universidad Politécnica de Madrid, Campus de Montegangedo \\
28660 Boadilla del Monte (Madrid), Spain \\
Orcid 0000-0003-1893-4453.} 
\email{jesus.castro@upm.es}

\thanks{The first and third authors are partially supported by State Research Agency and European
	Regional Development Fund via the grants no. PID 2020-117868GB-I00 and PID 2022-142559NB-I00, and Maria de Maeztu Excellence Unit IMAG CEX2020-001105-M funded by MCIN/AEI/10.13039/ 501100011033/
The second author is partially supported by the grant PID2021-124157NB-I00 funded by MCIN/AEI/10.13039/501100011033/ ‘ERDF A way of making Europe’, Spain; and by Comunidad Aut\'{o}noma de la Regi\'{o}n de Murcia, Spain, within the framework of the Regional Programme in Promotion of the Scientific and Technical Research (Action Plan 2022), by Fundaci\'{o}n S\'{e}neca, Regional Agency of Science and Technology, REF, 21899/PI/22.}



\date{}

\begin{abstract}
We  prove an existence and uniqueness theorem about spherical helicoidal (in particular, rotational) surfaces with prescribed  mean or Gaussian curvature in terms of a continuous function depending on the distance to its axis. 
As an application in the case of vanishing mean curvature, it is shown that the well-known conjugation between the belicoid and the catenoid in Euclidean three-space extends naturally to the 3-sphere to their spherical versions and determine in a quite explicit way their associated surfaces in the sense of Lawson. 
As a key tool, we use the notion of spherical angular momentum of the spherical curves that play the role of profile curves of the minimal helicoidal surfaces in the 3-sphere. 
\end{abstract}

\maketitle



\section{Introduction}

The study of minimal surfaces is one the most important research fields in differential geometry.
Minimal surfaces locally minimize area
and are characterized as surfaces whose mean curvature vanishes everywhere.
It is particularly interesting to study minimal surfaces in $3$-manifolds of constant curvature, such as the Euclidean space $\R^3$, the hyperbolic space $\H^3$, and the sphere $\s ^3$.

Minimal surfaces in $\R^3$ have been an intensively studied classic topic (see e.g.\ \cite{MP12} and references therein): 
Euler discovered in 1741 that a catenary rotating around a suitable axis provides a surface, called the {\em catenoid}, which minimizes area among surfaces of revolution after prescribing boundary values.
The catenoid is the only minimal surface of revolution, up to the plane (Bonnet, 1860).
The {\em helicoid} was first proved to be minimal by Meusnier in 1776. Catalan showed in 1842 that the helicoid, together with the plane, are the only ruled minimal surfaces.
We also emphasize the existence of a $1$-parameter family of minimal isometric surfaces (the so called {\em associated} surfaces) connecting the catenoid and the helicoid, which are said to be {\em conjugate} each other. The associated surfaces  are precisely 
the minimal helicoidal surfaces studied by Scherk in 1834.

In this paper we analyze in depth the possible spherical version of the above results. To do so, we will first have to determine which surfaces play the role of helicoids and catenoids in the 3-sphere.

Differently as happen in $\R^3$, 
interesting examples of {\em closed} minimal surfaces exist in $\s^3$ (see e.g.\ \cite{Br13b} and references therein): 
The simplest is the totally geodesic {\em equatorial 2-sphere}.
Almgren proved in 1966 that it is the only immersed minimal surface of genus $0$ in $\s^3$.
Another important example of a closed minimal surface in $\s^3$ is the {\em Clifford torus}.
Brendle proved in \cite{Br13a} that 
it is the only embedded minimal surface in $\s^3$ of genus $1$, thereby giving an affirmative answer to Lawson's Conjecture (cf.\ \cite{La70}).

In 1970, Lawson constructed in \cite[Theorem 3]{La70} an explicit infinite family of immersed minimal tori
in $\s^3$ which fail to be embedded, see \eqref{LawHelicoids}. 
We recall (\cite[Proposition 7.2]{La70}) that 
they are the only geodesically ruled minimal surfaces in $\s^3$ and so it is natural that they are  referred to as {\em Lawson spherical helicoids}.

To find the spherical version of the catenoids, it is therefore natural to concentrate on the rotational minimal surfaces of $\s^3$. 
Do Carmo and Dajczer coin the term ``catenoids" in \cite{dCD83} for those minimal rotational surfaces both in  $\H^3$  and in $\s ^3$, similar to what happens in $\R^3$. They recover the differential equation to be satisfied by the generatrix curve of the corresponding minimal surface of revolution and call it catenary. 
See also \cite[Theorem A]{Ri89}.
In \cite[Remark (3.34)]{dCD83}, it is commented that a conjugate correspondence can be established between the catenoids in $\s^3$  and the Lawson spherical helicoids. And, in \cite[Remark (3.35)]{dCD83}, it is mentioned that it is possible to determine explicitly the associated family of the catenoids in $\s^3$  but the authors will not go into that in the paper.

On the other hand, we must mention that  rotational minimal surfaces in $\s^3$ were also determined in a quite implicit way  in \cite{Ot70}. The constructed compact surfaces were named {\em Otsuki tori} by Penskoi in \cite{Pn13}, who showed that the Otsuki tori provide an extremal metric for the invariant spectral functional.
 In \cite{La70}, it is remarked that the Lawson helicoids are not the ones known to T.\ Otsuki \cite{Ot70} or E.\ Calabi \cite{Ca67}.

But we must also pay attention to a recent description of immersed rotationally symmetric minimal tori in $\s^3$ (see \cite[Theorem 1.4]{Br13b}, pointed out by Robert Kusner, according to the author). These surfaces are not embedded, but they turn out to be immersed in the sense of Alexandrov,
providing a large family of Alexandrov immersed minimal surfaces in $\s^3$. We will call them {\em Brendle-Kusner tori}.
Brendle shows in \cite{Br13c} that any minimal torus which is immersed in the sense of Alexandrov must be rotationally symmetric. 

Therefore, our aim in this article will be to fully justify which surfaces of $\s ^3$  deserve the appellation of {\em spherical catenoids} and why. In this way, we unify and clarify the different examples of rotational minimal surfaces in $\s^3$ that one can find in the literature. In addition, we will show in detail that the conjugation between the belicoid and the catenoid in $\R^3$ extends to $\s^3$. In fact, we will determine explicitly their associated family.

In order to achieve this objective, we deal in Section \ref{Sect:DefHel} with helicoidal surfaces in $\mathbb S^3$. This class includes the rotational ones. 
As a key tool, we introduce in Section \ref{Sect:Momentum} the notion of {\em spherical angular momentum} of a spherical curve, which completely determines it in relation with its relative position with respect to a fixed geodesic (see Theorem \ref{K determines}).
We compute it for the classical spherical catenaries studied by Bobillier and Gudermann in nineteenth century. Then we introduce a one-parameter family of \textit{spherical catenoids} in Definition \ref{def:catenoid} just as the rotational surfaces generated by these spherical catenaries. And they turn to be the minimal rotational surfaces in $\s^3$ considered in \cite{dCD83} and \cite{Ri89}.

The spherical angular momentum of the profile \textit{spherical} curve of a helicoidal surface
will play a fundamental role since it determines the geometry of the helicoidal surface, joint to its pitch, according to Corollary \ref{cor:Kkey}. 
So we are able to prove an existence and uniqueness result (Theorem \ref{Th:PrescribedHK}) showing that if we prescribe the mean or the Gaussian curvature of a spherical helicoidal (in particular, rotational) surface in terms of a continuous function depending on the distance to its axis,  once the pitch is set we get a one parameter family of helicoidal surfaces with this prescribed mean or Gaussian curvature determined by the spherical angular momenta of their profile curves.

As a first application, in Theorem \ref{Th:Rotmin} we perform in an easy direct way the local classification of the rotational minimal surfaces of $\s^3$ arriving at the spherical catenoids, together with the totally geodesic sphere and the Clifford torus that appear as limiting cases. Moreover, in Theorem \ref{Th:Rotmin}, we also identify both Otsuki and Brendle-Kusner tori with the \textit{compact} spherical catenoids, resulting the only  Alexandrov immersed minimal tori in $\s^3$.  

In Theorem \ref{Th:Cat to Hel}, we make a first approach to the assertion of \cite[Remark (3.34)]{dCD83} by showing explicitly that each spherical catenoid is locally isometric to two, and only two, spherical Lawson helicoids and, as expected, the Clifford torus appears as a self-conjugate surface (see Remarks \ref{h1 Clifford} and \ref{h1Cli}).

Finally, using our description of all the minimal helicoidal surfaces in $\mathbb S^3$ in terms of their spherical angular momenta, we prove in Theorem \ref{Th:minimal associated} that they are the {\em associated} surfaces  (in the sense of \cite{La70}) to the spherical catenoids and confirm the conjugation between spherical catenoids and Lawson spherical helicoids suggested in Theorem \ref{Th:Cat to Hel}. In this way, we carry out the statement of \cite[Remark (3.35)]{dCD83} in detail.

\medskip

{\em Acknowledgments.} The authors are very grateful to Marcos Dajczer, Jos\'{e} Miguel Manzano, Pablo Mira,  Miguel S\'{a}nchez, Francisco Torralbo and Francisco Urbano for helpful and fruitful discussions.



\section{Helicoidal and rotational surfaces in the 3-sphere} \label{Sect:DefHel}

Throughout this paper we will identify the 3-sphere $\s^3$ with the unit sphere in $\R^4$; that is,
$$
\s^3 = \{ (x_1,x_2,x_3,x_4)\in \R^4 : x_1^2+x_2^2+x_3^2+x_4^2=1 \}.
$$ 

A {\em helicoidal surface} in $\s^3$ is a surface invariant under the action of the helicoidal 1-parameter group of isometries given by the composition of a translation and a rotation in $\s^3$. Specifically, for any $a\in \R$, let us denote by $\psi_a(t)$ the translation along the great circle $c_0=\{ (x_1,x_2,0,0)\in \s^3 \}$ given by
$$
\psi_a(t)= \left(
\begin{array}{cccc}
\cos at & -\sin at & 0 & 0 \\
\sin at & \cos at & 0 & 0 \\
0 & 0 & 1 & 0 \\
0 & 0 & 0 & 1
\end{array}
\right),
$$
and for any $b\neq 0$, by $\phi_b(t)$ the rotation around $c_0$ given by
$$
\phi_b(t)= \left(
\begin{array}{cccc}
1 & 0 & 0 & 0 \\
0 & 1 & 0 & 0 \\
0 & 0 & \cos bt & -\sin bt \\
0 & 0 & \sin bt & \cos bt 
\end{array}
\right).
$$
We remark that the group $\phi_b(t)$ fixes the great circle $c_0$ and the orbits are circles centered on $c_0$.

Hence, a helicoidal surface (along the great circle $c_0$) is a surface invariant under $\psi_a(t)\circ \phi_b(t)$, for fixed {\em translation velocity} $a\in \R$  and {\em rotation velocity} $b\neq 0$. It can be locally parameterized by 
\begin{equation}\label{eq:Xab}
X_b^a(\xi)(s,t)=\psi_a(t) (\phi_b(t) \, \xi (s)),
\end{equation}
where $\xi : I \subseteq \R \rightarrow \s^2_+$ is a regular curve in the half totally geodesic sphere $\s^2_+ = \{ (x_1,x_2,x_3,0)\in \s^3 \, : x_3>0 \} $ of $\s^3$, called the {\em profile curve} (cf.\ \cite{MdS19}).

\begin{remark}\label{ab}
	Up to the change of parameter $\tilde t = b\, t$ we could consider simply $b=1$. In other words, for a helicoidal surface $X^a_b(\xi)$ the essential parameter is the ratio $h=a/b\in \R$ called the {\em pitch}. For this reason, we will simply denote $X^h=X^a_b$ when we choose $b=1$ and so $a=h$. In addition, 
	there is no restriction if we consider $a\geq 0$ and $b>0$; that is, $h\geq 0$.
\end{remark}

If $h=0$, $\psi_0$ is the identity and then $X^0(\xi)(s,t)=\phi_1(t) \, \xi (s)$ is a {\em rotational surface} with axis $c_0$ and the profile curve $\xi $ is its {\em generating curve}.

If $h=1$, the isometries $\psi_1(t)\circ \phi_1(t)$ are usually known as {\em Clifford translations}. In this case, the orbits are all great circles, which are equidistant from each other, and coincide with the fibers of the Hopf fibration $\pi \colon \s^3 \rightarrow \s^2 $. 

If we put $\xi (s) =(x(s),y(s),z(s),0)$, with $x(s)^2+y(s)^2+z(s)^2=1$ and $z(s)>0$, $s\in I \subseteq \R $, we can write
\begin{equation}\label{eq:Xh}
\begin{array}{c}
X^h(\xi)(s,t):=X_1^h(\xi)(s,t)=\left( 
x(s) \cos ht  - y(s) \sin ht , \right. \vspace{0.1cm } \\ 
\left. x(s) \sin ht + y(s) \cos ht, 
 z(s) \cos t , z(s) \sin t
\right).
\end{array}
\end{equation}
When $h=0$ in \eqref{eq:Xh}, we arrive at  
\begin{equation}\label{eq:paramXrot}
X^0(\xi)(s,t)=\left( 
x(s) , y(s) , 
 z(s) \cos t , z(s) \sin t \right).
\end{equation}

It may be useful sometimes to write the profile curve $\xi $ in geographical coordinates $0< \varphi \leq \pi/2$, 
$-\pi < \lambda \leq \pi$, that is:
$$\xi(s)=(\cos \varphi (s)  \cos \lambda (s) , \cos \varphi (s)  \sin \lambda (s) , \sin \varphi (s),0 ).$$
In this way, writing the 3-sphere $\s^3=\{ (\omega_1,\omega_2)\in \C^2 : |\omega_1|^2+|\omega_2|^2=1 \}$, we get
\begin{equation}\label{eq:paramXgeo}
X^h(\xi)(s,t)=\left( 
\cos \varphi (s) \,e^{i(ht+\lambda (s))} , \sin \varphi (s)\, e^{it} \right).
\end{equation}

\begin{remark}\label{rem:axially}
	The rotational surfaces $X^0(\xi)$ given by \eqref{eq:paramXrot} are congruent to the {\em axially symmetric} (with respect to the geodesic $c_0$) surfaces in $\s^3$ considered in \cite{AL15}, \cite{Br13b}, \cite{Pr10} or \cite{Pr16}. However, they choose the plane curve $\alpha(s)=(x(s),y(s))$ contained in the unit disk as profile curve, and then $z(s)=\sqrt{1-|\alpha(s)|^2}$.
	Moreover, \eqref{eq:paramXgeo} is compatible with the concept of \textit{twizzler} considered in \cite[Definition 3.4]{E11}. 
	We also deduce from \eqref{eq:paramXgeo} that $X^1(\xi)(s,t)=e^{it} \left( 
	\cos \varphi (s) \,e^{i\lambda (s)} , \sin \varphi (s) \right) $ are the so-called {\em Hopf immersions}.
\end{remark}

\begin{example}\label{ex:great circle}
Given $\theta \in (0,\pi) $, let $\zeta_\theta $ be the great semicircle $\sin \theta \, x_2 = \cos \theta \, x_3$ in $\s^2_+$, $\theta $ being the angle between $c_0$ and $\zeta_\theta$. Its arc length parametrization is given by
$$\zeta_\theta (s)=(\cos s, \cos \theta \sin s, \sin \theta \sin s,0), \, 0<s<\pi. $$
In particular, by considering the great semicircle $\zeta_{\pi/2} $ orthogonal to $c_0$ as profile curve, we have that
\begin{equation}\label{LawHelicoids}
X^h(\zeta_{\pi/2})(s,t)=\left( \cos s \, e^{iht} , \sin s \, e^{it} \right). 
\end{equation}
\begin{enumerate}[\rm (i)]
\item If $h=0$, the rotational surface $X^0(\zeta_{\pi/2}) $ gives the totally geodesic $\s^2 \hookrightarrow \s^3$ given by $x_2=0$.
\item If $h> 0$, using \cite[Proposition 7.2]{La70}, $X^h(\zeta_{\pi/2})$ provide the only geodesically ruled minimal surfaces in $\s^3$, which will be referred to as {\em Lawson spherical helicoids}. Looking at \eqref{LawHelicoids}, when $h=m/k$ with $(m,k)\in\Z^+ \times \Z^+$, $(m,k)=1$, we recover the compact minimal surfaces $\tau_{m,k}$ with zero Euler characteristic  (non-orientable if and only if $mk$ is even) studied by Lawson in \cite[Section 7]{La70}. The only one embedded is the Clifford torus, corresponding to $h=1$. 
\end{enumerate}
\end{example}

\begin{example}\label{ex:parallel}
Given $\varphi_0 \in (0,\pi/2)$, let $\mu_{\varphi_0} $ be the small circle in $\s^2_+$ parallel to $c_0$ given by $x_3=\sin \varphi_0$. Using \eqref{eq:paramXgeo}, for any $h\geq 0$, $X^h(\mu_{\varphi_0})$ leads to the CMC {\em standard torus} $\s^1(\cos\varphi_0)\times \s^1\left( \sin \varphi_0 \right) \hookrightarrow\s^3$ given by $|\omega_1|=\cos \varphi_0$ and $|\omega_2|=\sin \varphi _0$. The Clifford torus corresponds to $\varphi_0=\pi/4$.
\end{example}

\begin{example}\label{ex:small circle}
Given $\delta >0$, let $\eta_{\delta}$ be the small circle  in $\s^2_+$ orthogonal to $c_0$ given by $x_2=\tanh \delta$, with constant curvature $\kappa_\delta\equiv \sinh \delta$.
Using \eqref{eq:paramXrot}, the rotational surface $X^0(\eta_{\delta})$ gives the totally umbilical 2-sphere 
$\s^2(R_\delta) \hookrightarrow\s^3$ given by $x_2=\tanh \delta$, $x_1^2+x_3^2+x_4^2=\sech^2 \delta=:R_\delta^2$.
If $\delta =0$, we recover the totally geodesic equatorial 2-sphere $x_2=0$.

\end{example}



\section{The spherical angular momentum of a spherical curve} \label{Sect:Momentum}

We introduce a smooth function associated to any spherical curve, which completely determines it (up to a family of distinguished isometries) in relation with its relative position with respect to the fixed geodesic $c_0$.

Let $\xi=(x,y,z): I\subseteq \R \rightarrow \s^2 \subset \R^3$ be a spherical smooth curve parametrized by the arc length,
i.e.\ $|\xi (s)|=|\dot \xi (s)|=1, \, \forall s\in I$, where $I$ is some interval in $\R$. 
We will denote by a dot $\dot \,$ the derivative with respect to $s$ and by $\langle
\cdot, \cdot \rangle $ and $\times$ the Euclidean inner product and the cross product in $\R^3$ respectively. 

Let $T=\dot \xi $ be the unit tangent vector and $N=\xi \times \dot
\xi$ the unit normal vector of $\xi$. If $\nabla$ is the connection in $\s^2$, the oriented geodesic curvature $\kappa$ of $\xi$
is given by the Frenet equation $\nabla_T T = \kappa N$. Hence, we have that
\begin{equation}\label{Frenetbis}
 \ddot \xi =-\xi + \kappa N, \quad \dot N=-\kappa \,  \dot \xi
\end{equation}
and so $\kappa = \det (\xi, \dot \xi, \ddot \xi)$.

We pay attention to the geometric condition that the curvature of $\xi$ depends on the distance to the fixed geodesic $c_0$ of $\s^2$ given by $z=0$.
So we can assume, at least locally, the condition $\kappa=\kappa(z)$.

At any given point $\xi (s)$ on the curve, we introduce the {\em spherical angular momentum} (with respect to the fixed geodesic $c_0$) $\mathcal K (s)$ as the signed volume of the parallelepiped spanned by the position vector $\xi (s)$, the unit tangent vector $T(s)$ and the unit vector ${\bf e_3}:=(0,0,1)$, orthogonal to $c_0$. Concretely, we define
\begin{equation}\label{spherical momentum}
\mathcal K(s) := - \det (\xi(s),T(s), {\bf e_3})=-\langle N(s), {\bf e_3} \rangle = \dot x(s) y(s) -x(s) \dot y(s) .
\end{equation}
In physical terms, as a consequence of Noether's Theorem (cf.\ \cite{A78}), $\mathcal K (s)$ may be described as the angular momentum of a particle of unit mass with unit speed and spherical trajectory $\xi (s)$. We point out that $\mathcal K$ is a smooth function that takes values in $[-1,1]$. It is well defined, up to sign, depending on the orientation of the normal to $\xi$.

The spherical angular momentum of the great circles $\zeta_\theta$, $\theta \in (0,\pi)$, collected in Example \ref{ex:great circle}, is constant, concretely $\mathcal K_\theta=-\cos \theta$ and so, it distinguishes these geodesics by its relative position with respect to the fixed $c_0$. It is easy to check that the spherical angular momentum of the parallels $\mu_{\varphi_0}$, $\varphi_0\in (0,\pi/2)$, is also constant $-\cos \varphi_0$ (see Example \ref{ex:parallel}). Finally, for the small circles $\eta_\delta$, $\delta >0$, described in Example \ref{ex:small circle}, we obtain that $\mathcal K_\delta (s)= \tanh \delta \sin(\cosh \delta \, s)=\sinh \delta \, z_\delta(s) $, being $\sinh \delta$ precisely the curvature $\kappa_\delta$ of $\xi_\delta$ (see Section \ref{ex:Klinear} for details).

We now prove the main local result of this section, which 
shows how the spherical angular momentum $\mathcal K = \mathcal K (z) $ determines uniquely the spherical curve $\xi =(x,y,z)$, assuming $z$ non-constant.

\begin{theorem}\label{K determines}
Any spherical curve $\xi=(x,y,z):I\subseteq \R \rightarrow \s^2$, with $z$ non-constant, is uniquely determined by its spherical angular momentum $\mathcal K$ as a function of its coordinate $z$, that is, by $\mathcal K= \mathcal K(z)$. The uniqueness is modulo rotations around the $z$-axis. Moreover, the curvature of $\xi $ is given by $\kappa (z)=\mathcal K' (z)$.
\end{theorem}
\begin{proof}
Let $\xi=(x,y,z):I\subseteq \R \rightarrow \s^2$ be a unit speed spherical curve with $z$ non-constant, and assume that $\kappa=\kappa(z)$.
Using (\ref{Frenetbis}) and \eqref{spherical momentum}, we have that 
$\dot {\mathcal  K} = - \langle \dot N, {\bf e_3} \rangle = \kappa \langle \dot \xi, {\bf e_3} \rangle = \kappa \dot z $ and taking into account the assumption $\kappa=\kappa(z)$, we finally arrive at
\begin{equation}\label{anti K}
d\mathcal K= \kappa (z) dz,
\end{equation}
that is,  $\mathcal K (z)$ can be interpreted as an anti-derivative of $\kappa (z)$ or, equivalently, $\kappa (z)=\mathcal K ' (z)$.

We now use geographical coordinates in $\s^2$ and write 
$$\xi=(\cos \varphi  \cos \lambda , \cos \varphi  \sin \lambda , \sin \varphi ), \ -\pi/2 \leq \varphi \leq \pi/2, \, -\pi < \lambda \leq \pi .$$ 
Notice that the latitude $\varphi=\arcsin z$ is just the signed distance to $c_0$.
Using \eqref{spherical momentum}, $\mathcal K$ may be written as
\begin{equation}\label{K momentum}
\mathcal K (s)= - \dot \lambda (s)  \cos^2 \varphi (s).
\end{equation}
The unit-speed condition on $\xi$ implies that $\dot \varphi^2  + \dot \lambda^2  \cos^2 \varphi = 1$ and, since $\varphi$ is non constant (because we assume that $z$ is non constant) and using \eqref{K momentum}, we deduce that
\begin{equation}\label{dif}
ds=\frac{d\varphi}{\sqrt{1-\dot \lambda ^2  \cos^2 \varphi}}=\frac{\cos \varphi \, d\varphi}{\sqrt{\cos^2 \varphi-\mathcal K^2}}=\frac{dz}{\sqrt{1-z^2-\mathcal K^2}}
\end{equation}
and 
\begin{equation}\label{long}
d\lambda= - \frac{\mathcal K ds}{\cos^2 \varphi}=  \frac{\mathcal K \, ds}{z^2-1} .
\end{equation}
Hence, given $\mathcal K=\mathcal K (z)$ as an explicit function, looking at \eqref{dif} and \eqref{long}, one may compute $z(s)$ (and so $\varphi (s)$) and $\lambda (s) $ in three steps: integrate \eqref{dif} to get $s=s(z)$, invert to get $z=z(s)$, and integrate \eqref{long} to get $\lambda =\lambda (s)$. We observe that the integration constants appearing in \eqref{dif} and \eqref{long} simply mean a translation of the arc parameter and a rotation around the $z$-axis respectively. 
\end{proof}

\begin{remark}\label{c}
{\rm Following the proof of Theorem \ref{K determines},
we point out that we can determine by quadratures in a constructive explicit way the spherical curves such that $\kappa=\kappa(z)$, similarly as in Theorem 3.1 in \cite{S99}.
Concretely, if we prescribe a continuous function $\kappa=\kappa(z)$ as curvature, the proof of Theorem~\ref{K determines} leads to the computation of three quadratures, following the sequence:
\begin{enumerate}
\item[\rm (i)] A one-parameter family of anti-derivatives of $\kappa (z)$:
$$
\int \! \kappa (z) dz = \mathcal K(z).
$$
\item[\rm (ii)] Arc-length parameter $s$ of $\xi =(x,y,z) $ in terms of $z$, defined ---up to translations of the parameter--- by the integral:
$$
s=s(z)=\int\!\frac{dz}{\sqrt{1-z^2-\mathcal K(z)^2}},
$$
where $\mathcal K(z)^2 + z^2 < 1 $, and
inverting $s=s(z)$ to get $z=z(s)$ and the latitude 
$
\varphi(s)=\arcsin z(s)
$.
\item[\rm (iii)] Longitude of $\xi=(\cos \varphi  \cos \lambda , \cos \varphi  \sin \lambda , \sin \varphi )$ in terms of $s$, defined ---up to a rotation around the $z$-axis--- by the integral:
$$
\lambda (s)=\int \! \frac{\mathcal K(z(s))}{z(s)^2-1}ds ,
$$
where $|z(s)|<1$.
\end{enumerate}
We remark that we get a one-parameter family of spherical curves satisfying $\kappa=\kappa(z)$ according to the spherical angular momentum $\mathcal K (z)$ chosen in {\rm (i)} and verifying $\mathcal K(z)^2 +z^2 < 1 $. It will distinguish geometrically the curves according to their relative position with respect to the equator $c_0$ (or the $z$-axis).
}
\end{remark}

We show a simple example applying steps {\rm (i)-(iii)} of Remark~\ref{c}:

\begin{example}[$\kappa\!\equiv\! 0$]\label{ex:Kcte}
{\rm Then $\mathcal K \!\equiv \! c \!\in\! \R$, $s\!=\arcsin \frac{z}{\sqrt{1-c^2}} $, with $|c|<1$.
So $z(s)=\sqrt{1-c^2} \sin s $, $\lambda(s)\!=\!-\!\arctan (c \tan s)$ and, finally, $\xi (s)\!=\!(\cos s,-c \sin s, \sqrt{1-c^2} \sin s) $, which corresponds to the great circle $\sqrt{1-c^2}\,y+c\,z=0 $.
Up to rotations around the $z$-axis, they provide arbitrary great circles in $\s^2$, except the equator. Putting $c=-\cos \theta$, we recover the $\zeta_\theta$ given in Example \ref{ex:great circle}. As a consequence of Theorem~\ref{K determines}, the great circle $\zeta_\theta \equiv \sin \theta \, y = \cos \theta \, z$ is the only spherical curve (up to rotations around the $z$-axis) with constant spherical angular momentum $\mathcal K\!\equiv\! -\cos \theta$. 
}
\end{example}

\subsection{Spherical small circles}\label{ex:Klinear}
We assume that $\kappa \equiv k_0 >0 $ and we apply Remark~\ref{c}.
Then $\mathcal K (z)= k_0 z+c$ and, in this case, it is not difficult to get that 
$$z(s)=\frac{1}{1+k_0^2}\left( \sqrt{1-c^2+k_0^2}\, \sin \left(\sqrt{1+k_0^2}\, s\right)-c \, k_0\right)$$
with $|c|<\sqrt{1+k_0^2}$. 
But the computation of $\lambda$ is far from being trivial and depends on the values of $c$. After a long computation, we deduce:
\begin{itemize}
	\item If $|c|\neq k_0$:
		$\lambda(s)=\arctan \left(\frac{\sqrt{1 - c^2 + k_0^2} + (1 - c k_0 + k_0^2) \tan({\frac12 \sqrt{1 + k_0^2} s})}{(k_0-c)\sqrt{1 + k_0^2}}\right)+$\\
		
		\hspace{2.5cm}$+\arctan\left(\frac{\sqrt{1 - c^2 + k_0^2} + (1 + c k_0 + k_0^2) \tan({\frac12 \sqrt{1 + k_0^2} s})}{(k_0+c)\sqrt{1 + k_0^2}}\right).$
	\item If $c=k_0$:
	$\lambda(s)= \arctan\left( \frac{1 - (1 + 2 k_0^2) \tan(\frac12 \sqrt{1 + k_0^2} s)}{2 k_0 \sqrt{1 + k_0^2}}\right).$
	\item If $c=-k_0$:
	$\lambda(s)= \arctan\left( \frac{1 + (1 + 2 k_0^2) \tan(\frac12 \sqrt{1 + k_0^2} s)}{2 k_0 \sqrt{1 + k_0^2}}\right).$
\end{itemize}
Of course, up to rotations around the $z$-axis, we get all the non-parallel small circles of $\s^2$.
The parameter $c$ distinguishes the position of the circle with respect to the equator.
If $0\leq |c|<1$ the circles intersect the equator transversely; in particular, when $c=0$ we obtain the orthogonal circles to the equator. 
If $c=\pm 1$, the circles are tangent to the equator.
Finally, if $1<|c|<\sqrt{1+k_0^2}$, the circles do not intersect the equator (see Figure~\ref{Circles}).

\begin{figure}[h!]
\begin{center}
\includegraphics[height=3.4cm]{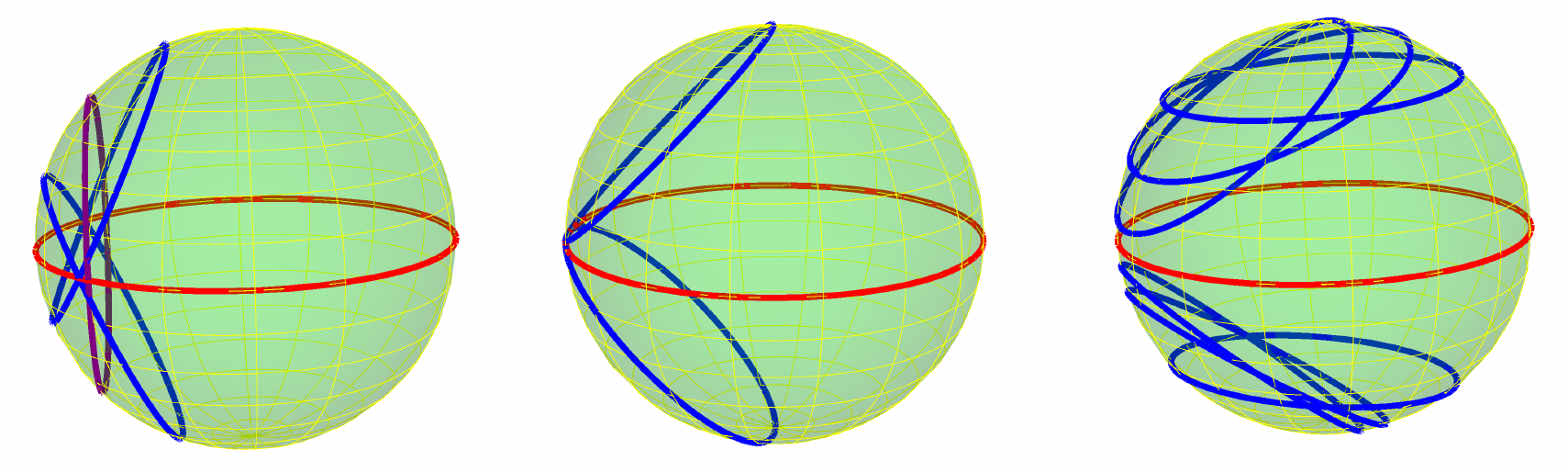}
\caption{ 
Small circles: $\mathcal K (z)= k_0 \, z + c, \ k_0 >0$; \newline $0\leq |c|<1$ (left), $c=\pm 1$ (center), $1<|c|<\sqrt{1+k_0^2}$ (right).}
\label{Circles}
\end{center}
\end{figure}

For instance, if $c=0$ and writing  $k_0=\sinh \delta$, we arrive (up to rotations around the $z$-axis)  at the small circle $\eta_\delta \equiv y=\tanh \delta$ of Example \ref{ex:small circle}.
Theorem~\ref{K determines} ensures that $\eta_\delta$ is the only spherical curve, up to rotations around the $z$-axis, with spherical angular momentum $\mathcal K (z)= \sinh \delta \,z$, $\delta >0$.

In conclusion, as a consequence of Theorem~\ref{K determines}, we have proved the following uniqueness result.
\begin{corollary}\label{cor:circles}
The spherical non-parallel circles of constant curvature $k_0\geq 0$ are the only spherical curves (up to rotations around the $z$-axis) with spherical angular momentum $\mathcal K(z)=k_0 z+ c$, $|c|<\sqrt{1+k_0^2}$.
\end{corollary}

\subsection{Spherical catenaries}\label{ex:Kinvz}

The spherical catenaries are the equilibrium lines of an inelastic flexible homogeneous infinitely thin massive wire included in a sphere, which is placed in a uniform gravitational field. Like any catenaries, their centres of gravity have the minimal altitude among all the curves with given length passing by two given points. They were studied by Bobillier in 1829 and by Gudermann in 1846. 

Using cylindrical coordinates $(r,\theta,z)$ in $\R^3$, they can be described analytically (cf.\ \cite{F93}) by the following first integral of the corresponding ordinary differential equation:
\begin{equation}\label{catenary}
(z-z_0) \, r^2 \, \frac{d\theta}{ds}= {\rm constant}.
\end{equation}

On the other hand, we study in this section spherical curves with curvature
\begin{equation}\label{k=a/z2}
\kappa(z)=c /z^2, \  c>0.
\end{equation}

We follow Remark \ref{c} considering (\ref{k=a/z2}) and choosing
\begin{equation}\label{eq:Kcat}
\mathcal K (z)=-\frac{c}{z}, \ c>0.
\end{equation}
We point out that it is the same choice of momentum as for the catenaries of $\R^2$ (see \cite{CCI16}) and of $\mathbb L^2$
(see \cite{CCI18}).
Then we have:
\begin{equation}\label{arc_catenary}
s=s(z)=\int \frac{z\,dz}{\sqrt{z^2(1-z^2)-c^2}},
\end{equation}
which implies that $c<1/2$ and $1-\sqrt{1-4c^2}<2z^2<1+\sqrt{1-4c^2}$, and hence:  
\begin{equation}\label{z_catenary}
z(s)=\sqrt{\frac{1+\sqrt{1-4c^2}\sin 2s}{2}}.
\end{equation}
We observe from \eqref{z_catenary} that the limit case $c=1/2$ leads to the parallel $z=1/\sqrt 2$.
In addition, we have from \eqref{eq:Kcat} that
\begin{equation}
\label{eq catenaries}
d\lambda=\frac{c}{z(1-z^2)}\,ds.
\end{equation}
Looking at (\ref{catenary}), taking into account that $r^2+z^2=1$ and $\theta = \lambda$, we deduce from (\ref{eq catenaries}) that we get a spherical catenary (with $z_0=0$ and constant $c\in (0,1/2)$). Combining (\ref{k=a/z2}) and (\ref{z_catenary}), we have that the intrinsic equation of the spherical catenaries is given by
$$ \kappa (s)=\frac{2c}{1+\sqrt{1-4c^2}\sin 2s}, \ 0<c<1/2. $$

Now we study when the spherical catenaries are closed curves. 
We put $2c =  \sin \beta$, with $0< \beta < \pi/2$, and call $\mathcal C_\beta$ the corresponding catenary. 
We have that \eqref{z_catenary} is rewritten as
\begin{equation}\label{z_catenary bis}
z_\beta(s)=\sqrt{\frac{1+\cos \beta \sin 2s}{2}}.
\end{equation}
Hence
$\sin (\beta / 2) < z(s) < \cos (\beta / 2)$. From \eqref{eq catenaries} and \eqref{z_catenary bis}, we deduce 
\begin{equation}\label{long_catenary}
\lambda_{\beta}(s)=\int_{0}^{s}\frac{\sqrt{2}\sin\beta \,dt}{(1-\cos\beta \sin(2t))\sqrt{1+\cos \beta\sin(2t)}}.
\end{equation}
Notice that $\lambda_\beta$ is an increasing function. Furthermore, since the function $z_{\beta}$ is $\pi$-periodic, the catenary $\mathcal C_{\beta}$ will be a closed curve if and only if  $\lambda_{\beta}(s+m\pi)=\lambda_{\beta}(s)+2k\pi$, with $m,k\in\Z$. As $\lambda_{\beta}'$ is also a $\pi$-periodic function,  $\lambda_{\beta}(s+m\pi)=\lambda_{\beta}(s)+m\lambda_{\beta}(\pi)$. Thus $\mathcal C_{\beta}$ is a closed curve if and only if
\begin{equation}\label{eq:Tcat}
T(\beta):=\frac{\lambda_{\beta}(\pi)}{2\pi}\in \Q.
\end{equation}
So the problem of being $\mathcal C_\beta$ a closed curve can be solved analyzing the function $T=T(\beta)$, $\beta \in (0,\pi/2)$, given in \eqref{eq:Tcat}. Using the same arguments as in \cite{Ot70}, \cite{AL15} or \cite{Pr16},
it can be proved that $T$ is a monotonically increasing function and $T((0,\frac{\pi}{2}))=(\frac{1}{2},\frac{\sqrt{2}}{2})$. Hence, for any $q\in \Q\cap(\frac{1}{2},\frac{\sqrt{2}}{2}) $, there exists an unique $\beta_q\in (0,\frac{\pi}{2})$, such as $\mathcal C_{\beta_q}$ is a closed catenary. 
Moreover, all these closed catenaries possess dihedral symmetry, i.e.\ the curve can be decomposed as the union of a fundamental piece and a finite number of rotations of the fundamental piece in the sphere.
The embeddedness of $\mathcal C_\beta$ would occur only if $T(\beta)=1/m$, for some positive $m\in \Z$. Thus the spherical catenaries are not simple curves. 
As a summary, the spherical catenaries consists of a sequence of undulations joining alternatively two parallels which are images of each other by rotations around the $z$-axis. They are either closed or dense in the zone between two parallels (see Figure \ref{Catenaries}). Using the nomenclature of \cite{AGM03}, we remark that the catenaries are generalized $\frac12$-elastic curves (see Proposition 9 in \cite{AGM03}), that is, critical points of the functional $\int_\xi \kappa^{1/2}\, ds$.
 
\begin{figure}[h]
\begin{center}
\includegraphics[height=6.5cm]{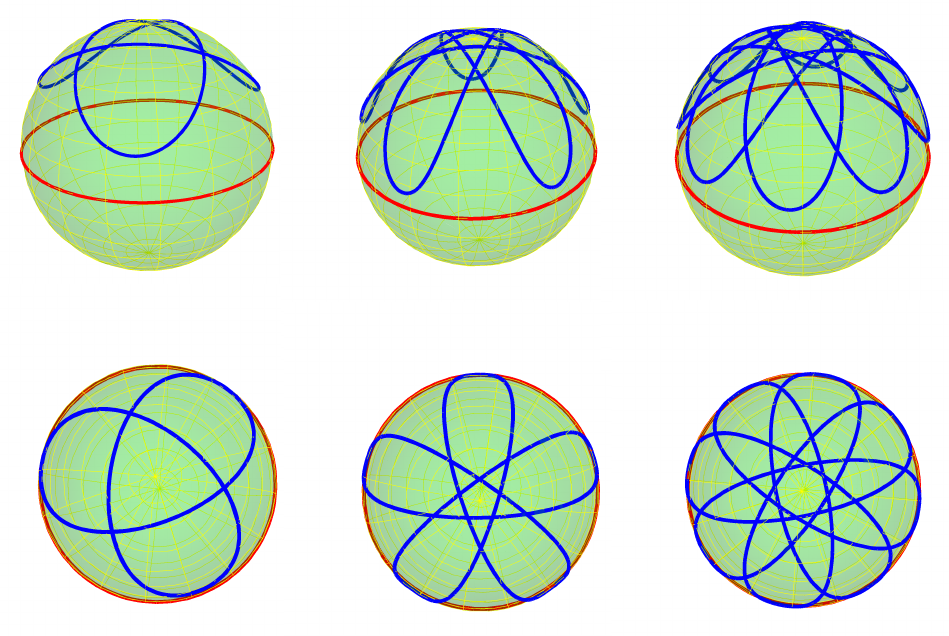}
\caption{ 
Two sights of the spherical catenaries $\mathcal C_{\beta_{2/3}}$ (left), $\mathcal C_{\beta_{3/5}}$ (center)
and $\mathcal C_{\beta_{5/8}}$ (right).
}
\label{Catenaries}
\end{center}
\end{figure}

In conclusion, as a consequence of Theorem \ref{K determines}, we have proved the following uniqueness result.
\begin{corollary}\label{cor:catenary}
The spherical catenaries $\mathcal C_\beta$, $0<\beta <\pi/2$,  are the only spherical curves (up to rotations around the $z$-axis) with spherical angular momentum $\mathcal K(z)=-\sin\beta/(2z)$. In addition, the closed catenaries  are non-embedded, possess dihedral symmetry and can be parametrized as $\mathcal C_{\beta_q}$, with $q\in  \Q\cap(\frac{1}{2},\frac{\sqrt{2}}{2}) $. 
\end{corollary}

\begin{definition}\label{def:catenoid}
We denote by $\mathrm{Cat}_\beta$ the rotational surface $X^0(\mathcal C_\beta)$ in $\s^3$ generated by the spherical catenary $\mathcal C_\beta$, $\beta \in (0,\pi/2)$. They will be referred to as {\em spherical catenoids}.
\end{definition}
They coincide with the minimal rotational surfaces studied by do Carmo and Dajczer in \cite{dCD83} and Ripoll in \cite{Ri89}. They will play a key role in Section \ref{sect:RotHelMin}.


\section{Prescribing curvature for a helicoidal surface in $\mathbb{S}^3$}\label{Sect:Geometry}

Following the ideas of Section \ref{Sect:DefHel}, three geometric elements determine a helicoidal surface in $\s^3$:
the translation $\psi_a(t)$ along the fixed geodesic $c_0$, the rotation $\phi_b(t)$ around $c_0$ and the profile spherical curve $\xi$. But, taking into account Remark \ref{ab}, only the pitch $h=a/b\geq 0$ and the profile curve $\xi $ are essential. As a consequence of Theorem \ref{K determines} and the fact that rotations around the $z$-axis are nothing but translations along $c_0$, we conclude immediately the following interesting result.
\begin{corollary}\label{cor:Kkey}
Given $h\geq 0$, any helicoidal surface $X^h(\xi) $ in $\s^3$ is uniquely determined, up to translations along $c_0$, by the spherical angular momentum $\mathcal K=\mathcal K(z)$ of its profile curve $\xi =(x,y,z)$, being $z>0$ non constant.
\end{corollary}

\begin{remark}\label{hxi}
Corollary \ref{cor:Kkey} means that, apart from the standard tori (see Example \ref{ex:parallel}) corresponding to $z$ constant, any helicoidal surface in the 3-sphere is characterized by its pitch $h=a/b\geq 0$ and the spherical angular momentum $\mathcal K=\mathcal K(z)$ of its profile curve. 

\noindent For instance, Lawson spherical helicoids $X^h(\zeta_{\pi/2})$ described in Example \ref{ex:great circle}, are uniquely determined by $h>0$ and null spherical angular momentum $\mathcal K\equiv 0$ (see Corollary \ref{cor:circles} and Example \ref{ex:Kcte}). 

\noindent On the other hand, the totally umbilical spheres $X^0(\eta_{\delta})$, $ \delta \geq 0$, of radius $R_\delta = \sech \delta$ (see Example \ref{ex:small circle}) are uniquely determined by $h=0$ and linear spherical angular momentum $\mathcal K(z)=k_0 z$, $k_0\geq 0$, being $k_0=\sinh \delta$ (see Corollary \ref{cor:circles} and Section \ref{ex:Klinear}).
\end{remark}

Consistent with Corollary \ref{cor:Kkey}, we are going to reveal that the geometry of the helicoidal surface $X^h(\xi)$, $h\ge 0$, 
i.e.\ its first and second fundamental forms, can be expressed in terms of the spherical angular momentum $\mathcal K=\mathcal K (z)$ of the profile curve $\xi =(x,y,z)$ and the non constant distance $z$ to $c_0$. We consider $\xi $ parameterized by the arc length $s$ and set $\xi (s) =(x(s),y(s),z(s),0)$, with $\dot x(s)^2+\dot y(s)^2+\dot z(s)^2=1$ and $z(s)>0$, $s\in I \subseteq \R $. 
For simplicity we write $X(s,t):=X^h(\xi)(s,t)$, see \eqref{eq:Xh}.

Denote by $g$ the induced metric of the helicoidal surface $X$. 
From \eqref{spherical momentum}, the entries of $g$ are given by
\begin{equation}\label{eq:metric g}
\begin{array}{c}
g_{11}=|X_s|^2=1, \quad 
 g_{12}=\langle X_s,X_t\rangle=-h \,\mathcal K(z),\vspace{0.5em} \\ 
\qquad  g_{22}=|X_t|^2=h^2 +(1-h^2)z^2.
\end{array}
\end{equation}
We have that $z^2+\dot z^2 +\mathcal K(z)^2 =1$, see \eqref{dif}.  Hence, a long straightforward computation provides that the unit normal $\nu$ of $X$ is given by
\begin{equation}\label{eq:normal}
\nu=\dfrac{1}{\alpha} \left(\nu_1,\nu_2,\nu_3,\nu_4 \right),
\end{equation}
 where
\[	
	\begin{array}{c}
		\nu_1=   z \left( -z(\dot x \sin ht+\dot y\cos ht)+\dot z(x\sin ht+y\cos ht) \right),\vspace{0.5em} \\ 
		\nu_2= z \left( z(\dot x \cos ht-\dot y\sin ht)+\dot z (-x\cos ht+y\sin ht ) \right), \vspace{0.5em} \\ 
		\nu_3=- z \mathcal K (z)\cos t  - h\dot z \sin t, \qquad 
		\nu_4=- z \mathcal K (z)\sin t  + h\dot z \cos t,
	\end{array}
\]
with
\begin{equation}\label{eq:alpha}
\alpha =\alpha (z):=\sqrt{h^2 \dot z^2 +  z^2}=\sqrt{ h^2(1-\mathcal K(z)^2) + (1-h^2) z^2 }>0.
\end{equation}
Notice that $\alpha^2=g_{11}g_{22}-g_{12}^2$.
Hence, since $\kappa = \det (\xi, \dot \xi, \ddot \xi)$, we get that the entries of the second fundamental form of $X$ are given by
\begin{equation}\label{eq:2ff hij}
\begin{array}{c}
\sigma_{11}=\langle X_{ss},\nu\rangle=\dfrac{ z k(z)}{\alpha (z)}=\dfrac{ z \mathcal K'(z)}{\alpha(z)},
\quad \sigma_{12}=\langle X_{st},\nu\rangle=\dfrac{h(1-\mathcal K(z)^2)}{\alpha(z)},\vspace{0.5em} \\ 
 \qquad \sigma_{22}=\langle X_{tt},\nu\rangle=\dfrac{(1-h^2)z^2\mathcal K(z)}{\alpha(z)}.
\end{array}
\end{equation}
From \eqref{eq:2ff hij}, we deduce that the mean curvature $H$ of the surface $X$ is given by
\begin{equation}\label{eq:mean-curvature}
2H(z)=
\dfrac{\left( h^2 + (1-h^2) z^2 \right) z\, \mathcal K'(z) + \left( 2 h^2(1-\mathcal K(z)^2) +(1-h^2) z^2 \right)\mathcal K(z)}
{\left( h^2(1-\mathcal K(z)^2) + (1-h^2) z^2 \right)^{3/2}}.
\end{equation}
And its Gauss curvature is $K_{\text{G}}(z)=1+K_{\text{ext}}(z)$, where the extrinsic curvature is 
\begin{equation}\label{eq:gauss-curvature}
K_{\text{ext}}(z)=\dfrac{(1-h^2) z^3\, \mathcal K (z) \mathcal K'(z)
-h^2 (1 - \mathcal K(z)^2)^2 }{\left( h^2(1-\mathcal K(z)^2) + (1-h^2) z^2 \right)^2}.
\end{equation}

\begin{remark}\label{rotPrinCur}
When $h=0$, using \eqref{eq:metric g}, \eqref{eq:2ff hij} and \eqref{eq:alpha}, we have that $g_{12}=0=\sigma_{12}$ and $\alpha (z)=z$; that is, the coordinate curves $X(s,\cdot)$ and $X(\cdot{,t})$ are curvature lines. Therefore, the principal curvatures of a rotational surface $X^0(\xi)$ in $\s^3$ are given by $\kappa_i=\sigma_{ii}/g_{ii}$, $i=1,2$, that is:
\begin{equation}\label{eq:princ curv}
\kappa_1(z)=\mathcal K'(z), \ \kappa_2(z)=\frac{\mathcal K(z)}{z}.
\end{equation} 
When $h=1$, using \eqref{eq:gauss-curvature}, we clearly have that $K_{\text{ext}}\equiv -1$ and then we obtain that the Hopf immersions $X^1(\xi)$ (see Remark \ref{rem:axially}) are flat, that is, $K_{\text{G}}\equiv 0$.
\end{remark}

Inspired by \cite{BK98}, we show in the next result that if we prescribe the mean curvature $H$ or the extrinsic curvature $K_{\text{ext}}$ (or, equivalently, the Gauss curvature $K_{\text{G}}$) of a helicoidal surface $X^h(\xi)$ by means of a function $H=H(z)$ or $K_{\text{ext}}=K_{\text{ext}}(z)$ respectively, we can determine the spherical angular momentum $\mathcal K(z)$ of the profile curve $\xi=(x,y,z)$ and, as a consequence of Corollary \ref{cor:Kkey}, the helicoidal surface $X^h(x,y,z)$ up to translations along the fixed geodesic $c_0$.
\begin{theorem}\label{Th:PrescribedHK}
\begin{enumerate}[\rm (a)]
\item Given $h\geq 0$, let $H=H(z)$, $z>0$, be a continuous function. 
Then there exists a one-parameter family of helicoidal surfaces 
with pitch $h$ in $\mathbb S^3$ and mean curvature $H(z)$, uniquely determined, up to translations along $c_0$,  by the spherical angular momenta of their profile curves given by
\begin{equation}\label{eq: mom-mean}
\mathcal K(z)=\pm \dfrac{\sqrt{h^2+(1-h^2)z^2}A(z)}{\sqrt{1+h^2 A(z)^2}},
\end{equation}
where
\begin{equation}\label{eq:mom-H} 
\displaystyle z^2 A(z)= 2 \! \int \! z H(z)dz.
\end{equation}
\item Given $h\geq 0$, $h\neq 1$, let $K_{\text{ext}}=K_{\text{ext}}(z)$, $z>0$, be a continuous function. 
Then there exists a one-parameter family of helicoidal surfaces 
with pitch $h$ in $\mathbb S^3$ and extrinsic curvature  $K_{\text{ext}}(z)$, uniquely determined, up to translations along $c_0$, by the spherical angular momenta of their profile curves given by
\begin{equation}\label{eq: mom-ext}
\mathcal K(z)
=\pm\sqrt{1+\dfrac{(1-h^2)z^2B(z)}{z^2+h^2 B(z)}},
\end{equation}
where 
\begin{equation}\label{eq:mom-KGauss}
B(z)=2 \!\int \! z  K_{\text{ext}}(z) dz .
\end{equation}
\end{enumerate}
\end{theorem}

\begin{remark}\label{cte int}
	The parameter in both uniparametric families described in Theorem \ref{Th:PrescribedHK} comes from the integration constant in \eqref{eq:mom-H} and \eqref{eq:mom-KGauss}.
\end{remark}

\begin{proof}
First, given $h\ge 0$ and a profile curve $\xi$, for any helicoidal surface $X=X^h(\xi)$ we define 
$$A(z):= \frac{\mathcal K(z)}{\left( h^2(1-\mathcal K(z)^2) + (1-h^2) z^2 \right)^{1/2}}, $$ where $\mathcal K(z)$ is the spherical angular momentum of $\xi$; note that 
$A(z)$ is well defined by \eqref{eq:alpha}.
Then we deduce that \eqref{eq:mean-curvature} can be written as
\begin{equation}\label{eq:ode-mean}
2H(z)=z A'(z)+2A(z).
\end{equation}
Assume now that the mean curvature of $X$ is prescribed by a continuous function $H=H(z)$, $z>0$. We can interpret \eqref{eq:ode-mean} as the linear ordinary differential equation 
$$  A'(z)+\frac{2}{z}A(z)=\frac{2H(z)}{z} $$
that $A=A(z)$ must satisfy, whose 
general solution is given in \eqref{eq:mom-H}. From the definition of $A=A(z)$, we arrive at  \eqref{eq: mom-mean},
which is also well defined by \eqref{eq:alpha}.

On the other hand, given $h\ge 0$ and a profile curve $\xi$, for any helicoidal surface $X=X^h(\xi)$ we consider  the well defined non positive function (see \eqref{eq:alpha}) given by
$$B(z):=-\frac{z^2(1-\mathcal{K}(z)^2)}{h^2(1-\mathcal K(z)^2) + (1-h^2) z^2} \leq 0.$$
Here, equation \eqref{eq:gauss-curvature} is rewritten as  
\begin{equation}\label{eq:ode-gauss}
B'(z)=2z K_{\text{ext}}(z).
\end{equation}
If we assume that the extrinsic curvature of $X$ is prescribed by a continuous function $K_{\text{ext}}=K_{\text{ext}}(z)$, we read
\eqref{eq:ode-gauss} with unknown  $B=B(z)$ and its solution is given in \eqref{eq:mom-KGauss}. Using the definition of $B=B(z)$, a direct computation gives us
 \eqref{eq: mom-ext} if $h\neq 1$, which is also well defined by \eqref{eq:alpha}.
\end{proof}


\section{Rotational and helicoidal minimal surfaces in $\s^3$}\label{sect:RotHelMin}

In order to study helicoidal (in particular, rotational) minimal surfaces in $\s^3$, given $h\geq 0 $, we prescribe $H\equiv 0$ in Theorem \ref{Th:PrescribedHK}.
Thus, $A(z)=c/z^2$, $c\in \R,$ and we get a one-parameter family of minimal helicoidal surfaces 
with pitch $h$, determined by the family of spherical angular momenta of \eqref{eq: mom-mean} given by
\begin{equation}\label{eq:K hel min}
\mathcal K^h_c (z)=\pm \frac{c \sqrt{h^2+(1-h^2)z^2}}{\sqrt{z^4+h^2 c^2}}.
\end{equation}
Recall from part {\rm (ii)} in Remark \ref{c} that it is necessary that $z^2+ \mathcal K^h_c (z)^2 < 1$.
Using \eqref{eq:K hel min}, it is not difficult to check that this condition is equivalent to $z^4+c^2<z^2$. Writing $z=\sin \varphi$, we have that $c^2<z^2(1-z^2)=\sin^2\varphi\cos^2\varphi=\frac14\sin^2 2\varphi <\frac14 $. In addition, from \eqref{eq:K hel min} it is clear that $\mathcal K^h_{-c}=-\mathcal K^h_c$ and then it is enough to consider $c\geq 0$. Hence, we take $c\in [0, 1/2)$.  We point out that we can choose the sign $\pm$ we want until we change the orientation.
We introduce now some notation.
For any $h\geq 0$ and $0\leq c<1/2$, we denote by  $\mathrm{Hel}_c^h$ the corresponding family of helicoidal minimal surfaces determined by \eqref{eq:K hel min} according to Theorem \ref{Th:PrescribedHK}, choosing minus sign. Summarizing this reasoning, we get the following result.

\begin{corollary}\label{cor:deform Hel}
Given $h\geq 0$, the helicoidal minimal surfaces of pitch $h$ in $\s^3$ are described by the one parameter family
$\mathrm{Hel}_c^h:= X^h (\xi_c^h)$, $0\leq c<1/2$, uniquely determined, up to translations along $c_0$, by the spherical angular momenta 
\begin{equation}\label{eq:K hel min-}
\mathcal K^h_c (z) = - \frac{c\, \sqrt{h^2+(1-h^2) z^2}}{\sqrt{z^4+h^2 c^2}}, \ 0\leq c<1/2,
\end{equation}
of their profile curves $\xi_c^h$.
\end{corollary}

\subsection{Rotational minimal surfaces in $\s^3$}\label{sect:RotMin}
We start studying the rotational case corresponding to $h=0$. From \eqref{eq:K hel min-}, we arrive at $\mathcal K^0_c(z)=-c/z$, with $c\in [0, 1/2)$. 
The previous study in Section \ref{ex:Kinvz} allows us to provide the following local classification theorem, where
we summarize and easily prove results from \cite{AL15}, \cite{Br13c}, \cite{dCD83}, \cite{Ot70}, \cite{Pr16} and \cite{Ri89} concerning minimal rotational surfaces in $\s^3$. In addition, from a global point of view, we also identify the Otsuki tori and the Brendle-Kusner tori and characterize them according to \cite{Br13c}.

\begin{theorem}\label{Th:Rotmin}
The only rotational minimal surfaces in $\s^3$ are open subsets of the following surfaces:
\begin{enumerate}[\rm (i)]
\item The totally geodesic sphere $\s^2 \hookrightarrow \s^3$.
\item The Clifford torus $\s^1(1/\sqrt 2)\times \s^1 (1/\sqrt 2) \hookrightarrow \s^3$.
\item The spherical catenoids $\mathrm{Cat}_\beta$, $0<\beta <\pi /2$ (see Definition \ref{def:catenoid}).
\end{enumerate}
Furthermore, the compact spherical catenoids $\mathrm{Cat}_{\beta_q}=X^0(\mathcal C_{\beta_q})$, $q\in \Q \cap (\frac{1}{2},\frac{\sqrt 2}{2})$ (see Corollary \ref{cor:catenary} and Definition \ref{def:catenoid}), are exactly the Otsuki tori, which also agree with the Brendle-Kusner tori. 
Joint with the Clifford torus, they are the only Alexandrov immersed minimal tori in $\s^3$.
\end{theorem}

\begin{remark}
{\rm
Brendle  showed, pointed by Kusner (see  \cite[Theorem 1.4]{Br13b}) that there exists an infinite family of minimal tori in $\s^3$ which are Alexandrov immersed, but fail to be embedded. In particular, we are going to prove that they are nothing but the so-called Otsuki tori (see \cite{Ot70,HS12}). 
Moreover, he proved in \cite{Br13c} that any minimal torus in $\s^3$ which is Alexandrov immersed must be rotationally symmetric.

\noindent On the other hand, Brendle solved Lawson Conjecture affirmatively proving in \cite{Br13a} that the Clifford torus is the only {\em embedded} minimal torus in $\s^3$. 
In this line, we conclude in Theorem \ref{Th:Rotmin} that the compact spherical catenoids, which coincide with the Otsuki and Brendle-Kusner tori, are the only {\em Alexandrov immersed} minimal tori in $\s^3$.
}
\end{remark}
We use \eqref{z_catenary bis} and \eqref{long_catenary} in the parametrization \eqref{eq:paramXrot} composed with a stereographic projection into $\R^3$ to visualize some of these tori in Figure \ref{Otsuki}.

\begin{figure}[h!]
\begin{center}
\includegraphics[height=4cm]{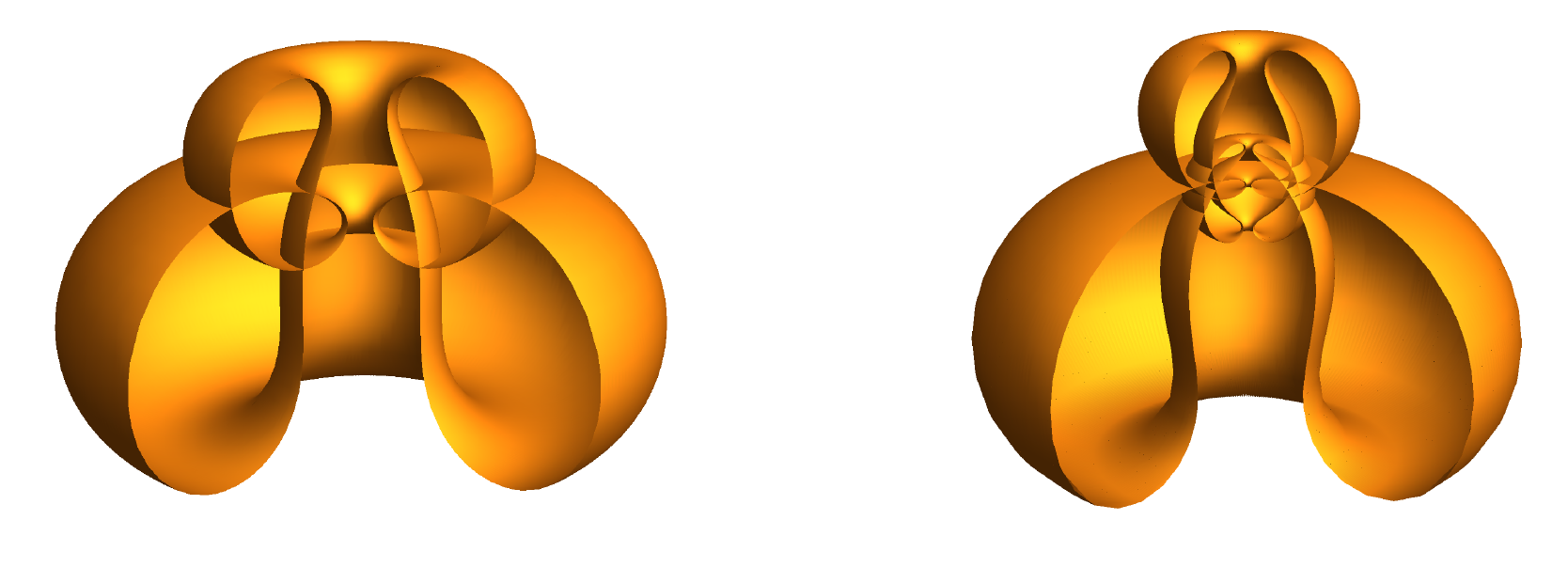}
\caption{ 
Open sights (with $t\in (0,3\pi/2)$) of the Otsuki-Brendle-Kusner spherical catenoids 
$\mathrm{Cat}_{\beta_{2/3}}$ 
and $\mathrm{Cat}_{\beta_{3/5}}$.
}
\label{Otsuki}
\end{center}
\end{figure}

\begin{proof}[Proof of Theorem \ref{Th:Rotmin}]
We are going to use Corollary \ref{cor:deform Hel} with $h=0$.
But first, we must have into account the case that $z$ is constant, see Corollary \ref{cor:Kkey}.
We know that this corresponds to the standard tori (see Example \ref{ex:parallel}), which are CMC surfaces and only the Clifford torus is minimal. This leads to case {\rm (ii)}. 

Next, we determine the surfaces $\mathrm{Hel}_c^0$ through  $\mathcal K^0_c(z)=-c/z$, with $c\in [0, 1/2)$.
If $c=0$, we get that $\mathcal K\equiv 0 $ and we arrive at case {\rm (i)}, see Remark \ref{hxi}. 
If $c\neq 0$, we proceed as in Section \ref{ex:Kinvz}, and
Corollary \ref{cor:catenary} and Definition \ref{def:catenoid} gives us case {\rm (iii)}, since 
$\mathrm{Cat}_\beta=X^0(\mathcal C_\beta)$, $\beta \in (0,\pi/2)$, is exactly $\mathrm{Hel}_c^0$, with $c=\frac12 \sin \beta $. 

On the other hand, Brendle-Kusner tori are described in the following way
(see \cite[Theorem 1.4]{Br13b} for more details).
Consider an immersion of the form
$$
F(t,\theta)=\left( r(t)\, e^{it}, \sqrt{1-r(t)^2}\, e^{i\theta} \right),
$$
where $r(t)$ is a smooth function taking values in the interval $(0,1)$ and satisfying the following differential equation for some constant $C\in (0,1)$:
\begin{equation}\label{eq:BrOt_tori}
\frac{r'^2}{r^4(1-r^2)}+\frac{1}{r^2 (1-r^2)}=\frac{4}{C^2}.
\end{equation}
It turns to be that $F$ is doubly periodic under certain rationality condition involving the constant $C$ and the interval $(1/2,\sqrt 2 /2)$.
Now we consider the spherical curve given by $\xi (t)= \left(r(t)\cos t, r(t) \sin t,\sqrt{1-r(t)^2}\right) $. 
We have that $ |\xi'|^2=\frac{r'^2+r^2(1-r^2)}{1-r^2}$. 
Now we write $r(t)=\cos \varphi (s)= \sqrt{1-z(s)^2}$ through the change of variable $t=\lambda (s)$ given by \eqref{eq catenaries}. Then it is not difficult to check that \eqref{eq:BrOt_tori} is satisfied if and only if $C=\sin \beta$, where $\beta \in (0,\pi/2)$ is the corresponding parameter of the spherical catenoid $\mathrm Cat_{\beta_q}$, and the aforementioned rationality condition on $C$ is exactly \eqref{eq:Tcat}.
The Clifford torus corresponds to the constant solution $r=z=1/\sqrt 2$, that is, the limiting case $\beta = \pi/2$.

To finish the proof, let us prove that Otsuki tori also coincide with the spherical catenoids described in Definition \ref{def:catenoid}. We follow the description of Otsuki tori given in \cite{HS12}.  For any $\alpha\in [0,2\pi)$, Otsuki tori are of the form
\[
\big(\mathcal{H(\theta)}\cos \alpha,
\mathcal{H(\theta)}\sin \alpha,
u(\theta),
v(\theta)\big),
\]
with $\mathcal{H(\theta)}:=\sqrt{1-h(\theta)^2-h'(\theta)^2}$, $u(\theta):=h(\theta)\sin\theta+h'(\theta)\cos\theta$, $v(\theta):=h'(\theta)\sin\theta-h(\theta)\cos\theta$
and where the support function $h=h(\theta) \in (0,1)$ is a non-constant periodic solution to the IVP given by
\[
2h(1-h^2)h'' +h'^2(1-h^2)(2h^2-1)=0, \quad 0<h(0)\leq \frac{1}{\sqrt{2}}, \quad h'(0)=0.
\]
We are able to find out a first integral to the above ODE. Concretely:
\begin{equation}\label{eq:1integral}
\frac{h(1-h^2-h'^2)}{\sqrt{1-h^2}}=c,
\end{equation}
for some constant $c\in\R$ depending on $h(0)$.
It is clear that we are dealing with $X^0(\xi )$, where $\xi = \xi (\theta)$
is the spherical curve given by $\xi(\theta)=(u(\theta),
v(\theta), \mathcal{H(\theta)})$.
We compute its spherical angular momentum and obtain that
\[
	\mathcal{K}(\theta)=-h(\theta) \sqrt{\frac{1-h(\theta)^2-h'(\theta)^2}{1-h(\theta)^2}}.
\]
Taking into account \eqref{eq:1integral} and that $z(\theta)=\mathcal{H(\theta)}=\sqrt{1-h(\theta)^2-h'(\theta)^2}$, we deduce that $\mathcal{K}(z)=-c/z$.  
Using Corollary \ref{cor:Kkey}, Corollary \ref{cor:catenary} and Definition \ref{def:catenoid}, we conclude the result.
\end{proof}

\subsection{Helicoidal minimal surfaces in $\s^3$}\label{sect:HelMin}
We go on by studying the proper helicoidal minimal surfaces $\mathrm{Hel}_c^h$ in $\s^3$, corresponding to $h>0$ (and $0\leq c < 1/2$) in Corollary \ref{cor:deform Hel}. We emphasize that, using Remark \ref{hxi}, they provide a deformation by helicoidal minimal surfaces of the Lawson spherical helicoids $\mathrm{Hel}_0^h$ corresponding to $c=0$. Moreover, one may easily check that for any $h > 0$, the Lawson helicoids  $\mathrm{Hel}_0^h$ and $\mathrm{Hel}_0^{1/h}$ are congruent and the Clifford torus corresponds to $\mathrm{Hel}_0^1$, see Example \ref{ex:great circle}-{\rm (ii)}.

On the other hand, recall that for any $\beta \in (0,\pi/2)$, $\mathrm{Cat}_\beta$ is the spherical catenoid 
$ X^0(\mathcal C_\beta)$ (see Definition \ref{def:catenoid}). Note that $\mathrm{Cat}_\beta = \mathrm{Hel}_{  c}^0$, when $  c=\frac12 \sin  \beta $.
%
%
In the following result we match spherical catenoids and Lawson spherical helicoids from a local isometric point of view. 
\begin{theorem}\label{Th:Cat to Hel}
Any spherical catenoid is locally isometric to two, and only two, Lawson spherical helicoids. 
Concretely, the catenoid $\mathrm{Cat}_\beta $, $\beta \in (0,\pi/2)$, is locally isometric to the helicoids 
$\mathrm{Hel}^{\tan (\beta/2)}_0$ (with pitch less than one) and $\mathrm{Hel}^{\cot (\beta/2)}_0$ (with pitch greater than one).
\end{theorem}

\begin{remark}
\label{h1 Clifford}
{\rm Recall that the limiting case $\beta = \pi/2$ degenerates into the Cliford torus $ \mathrm{Hel}^{1}_0 \equiv \mathrm{Cat}_{\pi/2} $.}
\end{remark}

\begin{proof}[Proof of Theorem \ref{Th:Cat to Hel}]
From formulae \eqref{eq:metric g} for the induced metric of a helicoidal surface and equations \eqref{z_catenary bis} and \eqref{LawHelicoids}, we may compute the first fundamental form of both $\mathrm{Cat}_\beta $ and $\mathrm{Hel}^h_0 $.
For $\mathrm{Cat}_\beta $, $\beta \in (0,\pi/2)$, the first fundamental form $I$, in coordinates $(s,t)$, is given by
\begin{equation}\label{eq:I Cat}
I=ds^2+\frac{1+\cos\beta \sin 2s}{2}dt^2,
\end{equation}
and for $\mathrm{Hel}^h_0 $, $h>0$, the first fundamental form $\tilde I$, in coordinates $(\tilde s,\tilde t)$, is
\begin{equation}\label{eq:I Hel}
\tilde I=d\tilde s^2+ (h^2 \cos^2 \tilde s + \sin^2 \tilde s) d\tilde t^2.
\end{equation}
For simplicity, we denote by $Y_\beta$ the corresponding parametrization of $\mathrm{Cat}_\beta$ and by $Y^h$ the one of $\mathrm{Hel}_0^h$.
We define the map $\Phi:=Y^h \circ \phi \circ (Y_\beta)^{-1}$ from $\mathrm{Cat}_\beta \equiv Y_\beta(s,t)$ to $\mathrm{Hel}^h_0\equiv Y^h(\tilde s,\tilde t) $, where
$$(\tilde s, \tilde t)=\phi (s,t)=\left(s\pm\frac{\pi}{4},\frac{t}{\sqrt{1+h^2}}\right),$$
the sign $\pm$ according to $h \lessgtr 1$.
Then, using \eqref{eq:I Hel}, we get that 
$$
\tilde I
=ds^2+ \frac{1}{2}\left( 1\pm \frac{1-h^2}{1+h^2} \sin 2s \right) d t^2.
$$
Looking at \eqref{eq:I Cat}, we must take 
\begin{equation}\label{eq:hbeta}
\cos \beta = \frac{1-h^2}{1+h^2} \ \text{ if }\ h<1, \quad\text{ or  }\quad \cos \beta = \frac{h^2-1}{1+h^2} \ \text{ if }\ h>1,
\end{equation}
in order to get $\tilde I = I$ and so $\Phi$ is the desired local isometry. The case $h=1$ corresponds to the Clifford torus (see Remark \ref{h1 Clifford}).
We point out that \eqref{eq:hbeta} provides a monotonically decreasing correspondence $h\in (0,1) \mapsto \beta \in(0,\pi/2)$,
and a monotonically increasing correspondence $h\in (1,\infty) \mapsto \beta \in(0,\pi/2)$,
whose inverse maps are precisely $h=\tan (\beta/2)$ and $h=\cot (\beta/2)$ respectively. This finishes the proof.
\end{proof}

We now aim to identify the surfaces $\mathrm{Hel}_c^h$, $0\leq c<1/2$, $h> 0$, $h\neq 1$, in terms of the {\em associated surfaces} (in the sense defined in \cite[Section 13]{La70}) of a given spherical catenoid $\mathrm{Cat}_\beta $, $\beta \in (0,\pi/2)$. 
For this purpose, recall the following fact, coming from the proof of \cite[Theorem 8]{La70}:

{\it For any minimal immersion $X\colon S \rightarrow \s^3$ of a simply connected surface $S$, there exists a differentiable, $2\pi$--periodic family of minimal isometric immersions $X_\theta \colon S \rightarrow \s^3$. Moreover, up to congruences, the maps $X_\theta$, $0\leq \theta \leq \pi$, represent (extensions of) all local, isometric, minimal immersions of $S$ into $\s^3$. The surface $X_{\pi/2} $ is called the conjugate surface of $X$.}

To introduce the family of associated surfaces in practice, we proceed as described in \cite{dCD83}:
{\it
Let $(x,y)$ be isothermal coordinates for the minimal immersion $X: S \rightarrow \s^3$. Denote by $I=E(dx^2+dy^2)$ and $I\!I=\sigma_{11}dx^2+2\sigma_{12}dxdy+\sigma_{22}dy^2$ the first and second fundamental forms of $X$, respectively. Set $\psi = \sigma_{11}-i\sigma_{12}$ and define a family of quadratic forms depending on a parameter $\theta \in [0,2\pi]$ by
$$
\sigma_{11}(\theta)= {\rm Re} (e^{i\theta}\psi), \ \sigma_{12}(\theta)= {\rm Im} (e^{i\theta}\psi), \
\sigma_{22}(\theta)= -{\rm Re} (e^{i\theta}\psi).
$$
Then $I_\theta=I$ and $I\!I_\theta= \sigma_{11}(\theta)dx^2+2\sigma_{12}(\theta)dxdy+\sigma_{22}(\theta)dy^2$ satisfy the Gauss and Codazzi equations, thus giving rise to the isometric family $X_\theta: S  \rightarrow \s^3$ of minimal immersions. The immersion $X_{\pi/2}$ is the conjugate immersion to $X_0=X$.
}

We are now able to prove in detail the following result commented in \cite[Remark 3.34]{dCD83}.

\begin{theorem}\label{Th:minimal associated}
Given a spherical catenoid 
$ \text{Cat}_\beta $,  $\beta \in (0,\pi/2)$, the associated immersions $( \text{Cat}_\beta)_\theta$, $0\leq \theta \leq \pi $, coincide with 
the helicoidal minimal immersions $\mathrm{Hel}_c^h$,  $0\leq c<1/2$, $h\geq 0$, $h\neq 1$, where
\begin{equation}\label{eq:deform min}
\frac{(1-h^2)^2 c^2+h^2}{(1+h^2)^2}=\frac{\sin^2 \beta}{4},
\qquad \cos \theta = \frac{c (1-h^2)  }{\sqrt{(1-h^2)^2 c^2}+h^2} .
\end{equation}

In particular, the conjugate surface of the spherical catenoid $\text{Cat}_\beta= (\text{Cat}_\beta)_0$ is the Lawson spherical helicoid $\mathrm{Hel}_0^h=(\text{Cat}_\beta)_{\pi/2}$, with 
$ \frac{h^2}{(1+h^2)^2}=\frac{\sin^2 \beta}{4}$.
\end{theorem}

\begin{remark}\label{coher}
{\rm 
Bearing in mind Theorem \ref{Th:Cat to Hel}, it is easy to check that $h^2/(1+h^2)^2=\frac14 \sin^2 \beta $ if, and only if, $h=\tan \beta/2$ or $h=\cot \beta/2$. In Figure \ref{graph deform}, we draw the curves $c_\beta$ given by
$$ c_\beta \equiv \frac{(1-h^2)^2 c^2+h^2}{(1+h^2)^2}=\frac{\sin^2 \beta}{4} $$
in the region  $0\leq c <1/2$, $h\geq 0$,
for different values of  $\beta \in (0,\pi/2)$. Each curve $c_\beta$ has two connected components separated by the line $h=1$.

\noindent It is not difficult to deduce from \eqref{eq:deform min} that $\theta=0 \Leftrightarrow c= \frac12 \sin \beta$, $h=0 $; $\theta \in (0,\pi /2) \Leftrightarrow h <1$; $\theta = \pi /2 \Leftrightarrow c=0 $, $h=\tan \beta/2$  or  $h=\cot \beta/2$; and $ \theta \in (\pi/2,\pi) \Leftrightarrow h>1$.
}
\end{remark}

\begin{figure}[h!]
	\begin{center}
		\includegraphics[height=6cm]{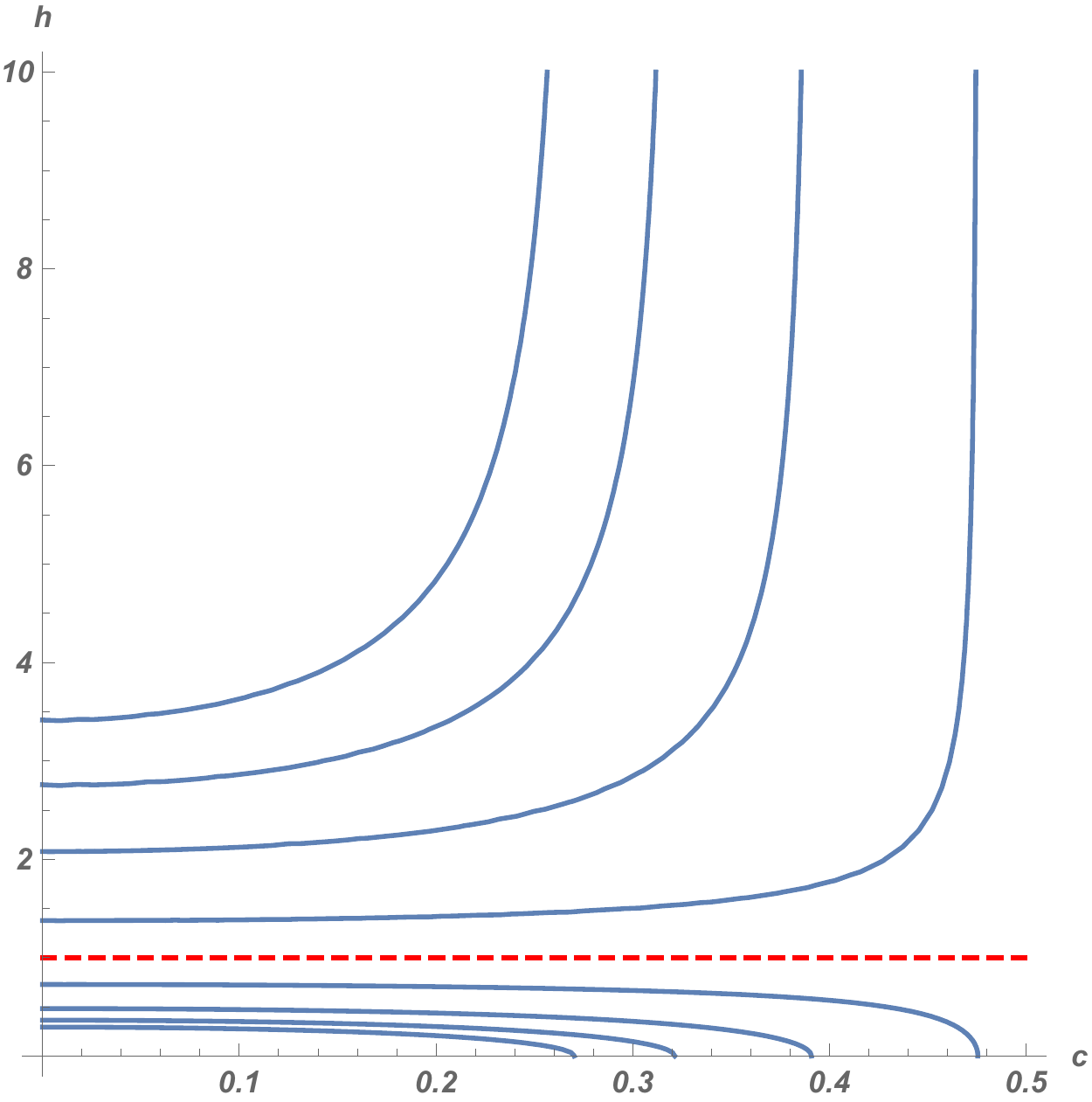}
		\caption{ 
			Representation of the associated immersions to a spherical catenoid.
		}
		\label{graph deform}
	\end{center}
\end{figure}

\begin{remark}\label{h1Cli}
	{\rm The immersions $\mathrm{Hel}_c^1$, $0\leq c<1/2$, are Hopf minimal surfaces in $\s^3$ (see Remark \ref{rotPrinCur}), so they are flat. Using a classical result in \cite{La69}, all of them provide representations of the Clifford torus $\mathrm{Hel}_0^1$ as helicoidal surfaces.
	}
\end{remark}

\begin{proof}[Proof of Theorem \ref{Th:minimal associated}]
Let $\mathrm{Cat}_\beta $, $\beta \in (0,\pi/2)$, be  the  spherical catenoid given in Definition~\ref{def:catenoid} and let $(\mathrm{Cat}_{\beta})_\theta$, $\theta \in (0,\pi)$, be its associated immersions. Let $(\tilde s, \tilde t)$ be the local coordinates for the catenoid $\mathrm {Cat}_{\beta}$ given by \eqref{eq:paramXrot} when the generatrix curve is the spherical catenary $\mathcal C_\beta $ studied in Section \ref{ex:Kinvz}.
In order to define the associated immersions   as in \cite[Section 13]{La70} 
we  take isothermal coordinates $(x,y)$ defined by
\begin{equation}\label{eq:conpar}
dx=\frac{1}{\tilde z(\tilde s) }d\tilde s, \quad  dy=d\tilde t,
\end{equation}
where $\tilde z= \tilde z_\beta $ is given by \eqref{z_catenary bis}.
In such new coordinates, using \eqref{eq:metric g}, \eqref{eq:2ff hij} and \eqref{eq:Kcat}, the coefficients of the first fundamental form $\widetilde{g}$ and the second fundamental form $\widetilde{\sigma}$ of $(\mathrm{Cat}_{\beta})_\theta$, $0\leq\theta<\pi$, are expressed by
\begin{equation}\label{eq:1ff}
\widetilde g_{xx}=\widetilde g_{yy}=\tilde{z}^2, \quad \widetilde g_{xy}=0,
\end{equation}
and
\begin{equation}\label{eq:2ff}
\widetilde \sigma_{xx}= \frac12 \sin \beta \cos \theta, \quad \widetilde \sigma_{xy}=\frac12 \sin \beta \sin\theta,  \quad \widetilde \sigma_{yy}=-\frac12 \sin \beta\cos\theta.
\end{equation}

Let now $(s,t)$ be the local coordinates for the helicoidal surfaces $\mathrm{Hel}_c^h$, $0\leq c<1/2$ and $h\geq 0$, $h\neq 1$, given by \eqref{eq:Xh} when the profile curves are the spherical curves $\xi^h_c$ described in Corollary \ref{cor:deform Hel}.

Next we are going to use coordinates $(\tilde z,\tilde t)$ for $(\mathrm{Cat}_{\beta})_\theta$ and $(z,t)$ for $\mathrm{Hel}_c^h$, where $\tilde z$ and $z$ come from \eqref{dif} when $\mathcal{K}(\tilde z)$ and $\mathcal{K}(z)$ are respectively given in \eqref{eq:Kcat} and \eqref{eq:K hel min-}.

We define a map $\Phi$ given by $(\tilde z,\tilde t)\mapsto (z,t)$ from  $(\mathrm{Cat}_{\beta})_\theta$  onto $\mathrm{Hel}_c^h$  such that
\begin{equation}\label{eq:Phi1}
\tilde z(z)=\sqrt{\frac{(1-h^2)z^2+h^2}{1+h^2}}
\end{equation}
and
\begin{equation}\label{eq:Phi2+}
d\tilde t=\frac{hc\sqrt{1+h^2}}{z \,\sqrt{(1-h^2)z^2+h^2}\, \sqrt{z^2-z^4-c^2}}\,dz+\sqrt{1+h^2}\,dt,\ \text{if $0<h<1$},
\end{equation}
or
\begin{equation}\label{eq:Phi2-}
	d\tilde t=-\frac{hc\sqrt{1+h^2}}{z \,\sqrt{(1-h^2)z^2+h^2}\, \sqrt{z^2-z^4-c^2}}\,dz-\sqrt{1+h^2}\,dt,\ \text{if $h>1$}.
\end{equation}

 %
All the above changes of local parameters are summarized in the following diagram:
\begin{equation}\label{eq:diagram}
	\xymatrix{
		(x,y) \ar[d]^{(X^0(\mathcal C_\beta))_\theta} \ar[r]^{\eqref{eq:conpar}} &  (\tilde s, \tilde t) \ar[r]^{\eqref{dif}}_{\eqref{eq:Kcat}} & (\tilde z, \tilde t) \ar[r]^{\Phi}_{
			\begin{minipage}{1cm}
				\centering \scriptsize \eqref{eq:Phi1} \\ \eqref{eq:Phi2+} \\ \eqref{eq:Phi2-}
			\end{minipage}
		}
		 & (z,t) \ar[r]^{\eqref{dif}}_{\eqref{eq:K hel min-}} & (s,t) \ar[d]^{X^{h}(\xi^h_c)}\\
		(\mathrm{Cat}_\beta)_\theta &                     &                             &             & \mathrm{Hel}_c^h	
	}
\end{equation}

Our aim is to obtain the coefficients of the first fundamental form $g$ and the second fundamental form $\sigma$ of $\mathrm{Hel}_c^h$ in the coordinates $(x,y)$ and  show that they coincide with~\eqref{eq:1ff} and~\eqref{eq:2ff} respectively when $h$, $c$, $\beta$ and $\theta$ satisfy~\eqref{eq:deform min}.

Assume first that $0<h<1$. From \eqref{eq:diagram} (using \eqref{eq:Phi2+}) and taking into account \eqref{eq:deform min}, we compute the partial derivatives of  $(x,y)$ with respect to $s$ and $t$,  obtaining
 \begin{align}\label{eq:partial xy}
  x_s=&\frac{dx}{d\tilde s}\frac{d\tilde s}{d\tilde z} \frac{d \tilde z}{dz}\frac
  {d z}{ds} = \frac{\sqrt{1+h^2}\,z^2}{\sqrt{h^2+(1-h^2)z^2}\,\sqrt{ c^2h^2+z^4}}, \nonumber
  \\
   x_t=&0,\nonumber \\  \\
  y_s&=\frac{dy}{d\tilde t}\frac{d\tilde t}{dz}\frac{dz}{ds} =\frac{ h c\sqrt{1+h^2} }{\sqrt{h^2+(1-h^2)z^2}\,\sqrt{ c^2h^2+z^4}},  \nonumber
  \\ y_t&=\frac{dy}{d\tilde t}\frac{d\tilde t}{dt}=\sqrt{1+h^2}.  \nonumber
 \end{align}

Denote by $g_{xx}$, $g_{xy}$ and $g_{yy}$ the coefficients of the first fundamental form of $\mathrm{Hel}_c^h$ through the parameters $(x,y)$ whilst $g_{ss}$, $g_{st}$, and $g_{tt}$ denote the entries of $g$ through the parameters $(s,t)$. A straightforward computation provides the following  relation between them:
\[\begin{cases}
	g_{ss}
	=x_s^2 \, g_{xx}+2x_sy_s\, g_{xy}+y_s^2\,g_{yy}\\
	g_{st}
	=x_sy_t\,g_{xy}+y_sy_t\,g_{yy}\\
	g_{tt}
	=y_t^2\, g_{yy}
\end{cases}\]
 Using~\eqref{eq:metric g} and \eqref{eq:partial xy}, we get that
\begin{eqnarray*}
g_{xx}&=&
\frac{1}{x_s^2}\big( g_{ss}-\frac{2y_s}{y_t}g_{st}+ \frac{y_s^2}{y_t^2}g_{tt}\big)=\frac{h^2+(1-h^2)z^2}{1+h^2},\\
g_{xy}&=&
\frac{1}{x_sy_t}\big(g_{st}- \frac{y_s}{y_t}g_{tt}\big)=0,\\
g_{yy}&=&
\frac{1}{y_t^2}g_{tt}=\frac{h^2+(1-h^2)z^2}{1+h^2}.
\end{eqnarray*} 
From \eqref{eq:Phi1}  we deduce that $g$ coincides with $\widetilde g$ as desired, see~\eqref{eq:1ff}. This ensures that the map $\Phi$ is a local isometry. To conclude the result, we now check that the second fundamental forms of $(\mathrm{Cat}_{\beta})_\theta$ and $\mathrm{Hel}_c^h$  also coincide. Denote by $\sigma_{xx}$, $\sigma_{xy}$, and $\sigma_{yy}$ the coefficients of the second fundamental form of $\mathrm{Hel}_c^h$ through the parameters $(x,y)$ whilst $\sigma_{ss}$, $\sigma_{st}$, and $\sigma_{tt}$ denote the entries of $\sigma$ through the parameters $(s,t)$. We arrive at the corresponding relations:
$$\begin{cases}
	\sigma_{ss}=x_s^2\, \sigma_{xx}+2x_sy_s\,\sigma_{xy}+y_s^2\,\sigma_{yy},\\
	\sigma_{st}=x_sy_t\, \sigma_{xy}+y_sy_t\,\sigma_{yy},\\
	\sigma_{tt}=y_t^2\, \sigma_{yy},
\end{cases}$$
Then, using \eqref{eq:2ff hij} and \eqref{eq:partial xy} taking \eqref{eq:deform min} into account,  we obtain that 
\begin{eqnarray*}
\sigma_{xx}&=&
\frac{1}{x_s^2}\big( \sigma_{ss}-\frac{2y_s}{y_t}\sigma_{st}+ \frac{y_s^2}{y_t^2}\sigma_{tt}\big)=c\,\frac{1-h^2}{1+h^2}=\frac 1 2 \sin\beta \cos\theta,\\
\sigma_{xy}&=&
\frac{1}{x_sy_t}\big(\sigma_{st}- \frac{y_s}{y_t}\sigma_{tt}\big)=\frac{ h}{1+h^2}=\frac{1}{2}\sin\beta\sin\theta, \\
\sigma_{yy}&=&
\frac{1}{y_t^2}\sigma_{tt}=-c\,\frac{1-h^2}{1+h^2}=-\frac{1}{2}\sin\beta\cos\theta.
\end{eqnarray*}
 Therefore, using the fact that surfaces are uniquely determined by their first and second fundamental forms, we can conclude that $(\mathrm{Cat}_{\beta})_\theta$ and $\mathrm{Hel}^h_c$ are congruent surfaces.
 
 Assume now that $h>1$; in that case, from \eqref{eq:Phi2-}, we have that 
 \begin{align}\label{eq:partial xy2}
 	x_s=& -\frac{\sqrt{1+h^2}\, z^2}{\sqrt{h^2+(1-h^2)z^2}\,\sqrt{ c^2h^2+z^4}} , \nonumber
 	\\
 	x_t=&0, \nonumber \\  \\
 	y_s&=-\frac{ h c\sqrt{1+h^2} }{\sqrt{h^2+(1-h^2)z^2}\,\sqrt{ c^2h^2+z^4}},  \nonumber
 	\\ y_t&=-\sqrt{1+h^2}. \nonumber
 \end{align}

Following similar computations using now \eqref{eq:partial xy2}, we reach the same expressions for the entries of the first and second fundamental forms that the corresponding to the case of $0<h<1$. Therefore  we conclude again that $(\mathrm{Cat}_{\beta})_\theta$ and $\mathrm{Hel}^h_c$ are congruent surfaces. 
\end{proof}

\begin{remark}\label{Pablo}
	In the same way that we can find compact catenoids $\mathrm{Cat}_{\beta_q} $, $q\in \Q\cap(\frac{1}{2},\frac{\sqrt{2}}{2}) $ (see Definition \ref{def:catenoid} and Corollary \ref{cor:catenary}), it is to be expected that the same happens in the family of minimal helicoids $\mathrm{Hel}_c^h$,  $0\leq c<1/2$, $h\geq 0$, $h\neq 1$.
	But it is not an easy problem to determine even when their $z$-function is periodic.  
\end{remark}


\end{document}